\documentclass[reqno,11pt]{amsart}

\usepackage[T1]{fontenc}
\usepackage[english]{babel}
\usepackage{amsmath}
\usepackage{amsthm}
\usepackage{amssymb}
\usepackage{amsfonts}
\usepackage{amsxtra}
\usepackage{color}
\usepackage[breaklinks]{hyperref}
\usepackage{enumerate}

\numberwithin{equation}{section}

\newcommand{\td}{\,\mathrm{d}}
\newcommand{\Ad}{\textup{Ad}}
\newcommand{\ad}{\textup{ad}}
\renewcommand\Re{\operatorname{Re}}

\newcommand{\id}{\textup{id}}
\newcommand{\tr}{\textup{tr}}

\newcommand{\GL}{\textup{GL}}

\newcommand{\SL}{\textup{SL}}

\newcommand{\PSL}{\textup{PSL}}

\newcommand{\SO}{\textup{SO}}
\newcommand{\so}{\mathfrak{so}}

\newcommand{\RR}{\mathbb{R}}

\newcommand{\CC}{\mathbb{C}}

\newcommand{\NN}{\mathbb{N}}

\newcommand{\HH}{\mathbb{H}}

\newcommand{\1}{\mathbf{1}}

\newcommand{\Ind}{\textup{Ind}}

\newcommand{\calF}{\mathcal{F}}
\newcommand{\calO}{\mathcal{O}}
\newcommand{\calS}{\mathcal{S}}

\newcommand{\calH}{\mathcal{H}}

\newcommand{\calU}{\mathcal{U}}

\newcommand{\calC}{\mathcal{C}}
\newcommand{\calR}{\mathcal{R}}

\newcommand{\frakg}{\mathfrak{g}}

\newcommand{\frakk}{\mathfrak{k}}
\newcommand{\frakp}{\mathfrak{p}}

\newcommand{\frakn}{\mathfrak{n}}
\newcommand{\fraka}{\mathfrak{a}}

\newcommand{\frakm}{\mathfrak{m}}

\newcommand{\frakt}{\mathfrak{t}}

\newcommand{\frakh}{\mathfrak{h}}

\newcommand{\Det}{\textup{Det}}

\renewcommand{\mod}{\textup{mod}}

\newcommand{\vol}{\textup{vol}}
\newcommand{\pr}{\textup{pr}}

\newcommand{\diag}{\textup{diag}}

\newcommand{\blank}{\cdot}

\newcommand{\aut}{\textup{\scriptsize aut}}
\renewcommand{\mod}{\textup{\scriptsize mod}}
\DeclareMathOperator{\supp}{supp}
\DeclareMathOperator{\Isom}{Isom}
\DeclareMathOperator{\Stab}{Stab}

\theoremstyle{plain}
\newtheorem{theorem}{Theorem}[section]
\newtheorem{proposition}[theorem]{Proposition}
\newtheorem{lemma}[theorem]{Lemma}
\newtheorem{corollary}[theorem]{Corollary}
\newtheorem{fact}[theorem]{Fact}
\newtheorem{thmalph}{Theorem}

\theoremstyle{definition}

\newtheorem{remark}[theorem]{Remark}

\begin{document}

\title[Restriction of automorphic forms on hyperbolic manifolds]{Estimates for the restriction of automorphic forms on hyperbolic manifolds to compact geodesic cycles}
\date{January 7, 2014}

\author{Jan M\"{o}llers}
\author{Bent \O rsted}

\address{Institut for Matematiske Fag, Aarhus Universitet, Ny Munkegade 118, 8000 Aarhus C, Denmark}
\email{moellers@imf.au.dk}

\address{Institut for Matematiske Fag, Aarhus Universitet, Ny Munkegade 118, 8000 Aarhus C, Denmark}
\email{orsted@imf.au.dk}
\begin{abstract}
We find estimates for the restriction of automorphic forms on hyperbolic manifolds to compact geodesic cycles. The geodesic cycles we study are themselves hyperbolic manifolds of lower dimension. The restriction of an automorphic form to such a geodesic cycle can be expanded into eigenfunctions of the Laplacian on the geodesic cycle. We prove exponential decay for the coefficients in this expansion.
\end{abstract}



\maketitle

\section*{Introduction}

Analysis on Riemannian locally symmetric spaces is a topic with relations to number theory, spectral theory, and representation theory. In this paper we shall introduce some new techniques from representation theory to give results about periods, in the sense of integrating automorphic forms over suitable submanifolds, namely totally geodesic submanifolds. Specifically, for hyperbolic locally symmetric spaces, we shall combine the following three topics:
\begin{enumerate}[(1)]
\item (Representation theory) Branching problems for spherical unitary representations and corresponding invariant bilinear forms,
\item (Number theory) Geodesic periods for automorphic functions,
\item (Spectral theory) Spectral asymptotics for restrictions of automorphic functions,
\end{enumerate}
where the main new ingredient is to be found under (1). Indeed, here we give a natural extension of techniques by J. Bernstein and A. Reznikov~\cite{BR04,Rez08} who considered invariant trilinear functionals; it turns out that their method can be modified and used in connection with the introduction of new invariant bilinear forms that were investigated by T. Kobayashi and B. Speh~\cite{KS}. Thus we can study (2) and (3) and obtain estimates for the asymptotic behaviour of Fourier coefficients, in particular an exponential decay.

Our techniques rely on rather explicit formulas for hyperbolic spaces; but in principle they extend to other locally symmetric spaces, as the invariant bilinear forms by T. Kobayashi and B. Speh~\cite{KS} have been established in great generality by the authors in joint work with Y. Oshima~\cite{MOO}.

\subsection{Geodesic cycles in hyperbolic manifolds}

Let $Y$ be a connected hyperbolic manifold of dimension $n$. Consider a compact geodesic cycle $Y'\subseteq Y$ of dimension $0<m<n$. This is a totally geodesic submanifold which itself is a compact hyperbolic manifold. For $m=1$ the submanifold $Y'$ is simply a closed geodesic.

Let $\phi$ be an automorphic form on $Y$, i.e. $\phi\in C^\infty(Y)$ is an $L^2$-eigenfunction of the Laplace--Beltrami operator $\Delta$ on $Y$. Consider the restriction of $\phi$ to the geodesic cycle $Y'$. The Laplace--Beltrami operator $\Delta'$ on $Y'$ is non-negative and has purely discrete spectrum on $L^2(Y')$ with finite multiplicities. Denote by $0=\lambda_0<\lambda_1\leq\lambda_2\leq\ldots$ its eigenvalues, counted with multiplicities, and by $(\phi_j)_j$ a corresponding orthonormal basis of $L^2(Y')$ consisting of eigenfunctions. Then the restriction of $\phi$ to $Y'$ has an expansion
\begin{equation}
 \phi|_{Y'} = \sum_{j=0}^\infty c_j\phi_j.\label{eq:FourierExpansion}
\end{equation}
We are interested in the behaviour of the coefficients $c_j$ as $j\to\infty$. The numbers $c_j$ also have an interpretation as period integrals over the geodesic cycle $Y'$:
\begin{equation*}
 c_j = \int_{Y'} \phi|_{Y'}\cdot\overline{\phi_j}.
\end{equation*}

\subsection{Estimates of periods}

We prove that the coefficients $c_j$ decay exponentially as $j\to\infty$. To make this precise we define numbers $b_j\geq0$ by
\begin{equation*}
 b_j = |c_j|^2 e^{\pi\sqrt{\lambda_j}}
\end{equation*}
Then our main result can be stated as follows:

\begin{thmalph}\label{thm:MainThm}
There exists a constant $C>0$ such that for $T\geq1$:
\begin{equation}
 \sum_{|\lambda_j|\leq T} b_j \leq CT^{\frac{2n-m-3}{2}}.\label{eq:MainThmEstimate}
\end{equation}
\end{thmalph}



We remark that for the case $m=1$ Theorem~\ref{thm:MainThm} provides estimates for the Fourier coefficients of automorphic forms along closed geodesics. In particular, for $n=2$ we obtain estimates for the Fourier coefficients along closed geodesics of automorphic forms on hyperbolic surfaces such as the modular surface.

\subsection{Relation to other results}

Estimates for the restriction of eigenfunctions of the Laplacian have been studied in various different settings. General estimates for the restriction from compact Riemannian manifolds to arbitrary submanifolds were obtained by Burq--G\'{e}rard--Tzvetkov~\cite{BGT07}. In contrast to our estimates for the growth of Fourier coefficients of the restriction, they estimate the $L^2$-norm of the restriction of an eigenfunction in terms of its eigenvalue. In this respect our results are of a different nature than theirs.

In \cite{Rez04,Rez13} Reznikov obtains results similar to those of Burq--G\'{e}rard--Tzvetkov for closed geodesics in hyperbolic surfaces, i.e. $n=2$ and $m=1$ in our setting. His results can be viewed as a refinement of \eqref{eq:MainThmEstimate}, namely he provides a uniform bound for the constant $C$ in terms of the eigenvalue of $\phi$. In fact, the statement \cite[Theorem B]{Rez04} implies our Theorem~\ref{thm:MainThm} in this special case.

We further remark that some geodesic periods are closely related to special values of $L$-functions. Watson~\cite{Wat02} discovered a relation between trilinear periods and special values of triple $L$-functions for $\SL(2,\RR)$. This relation was used by Bernstein--Reznikov~\cite{BR10} to obtain subconvexity bounds for these $L$-functions. A similar relation between geodesic periods and $L$-functions for orthogonal groups was conjectured by Ichino--Ikeda~\cite{II10}, based on a conjecture by Gross--Prasad~\cite{GP92}. It would be interesting to investigate the connection between our estimates for geodesic periods and subconvexity bounds for these $L$-functions.

For more geometrical and cohomological aspects of geodesic cycles in locally Riemannian symmetric spaces for orthogonal groups see e.g. Kobayashi--Oda~\cite{KO98} and Bergeron~\cite{Ber06} and references therein. We also refer the reader to the recent survey paper by Schwermer~\cite{Sch10}.

\subsection{Strategy of proof}

We can identify $Y\cong\Gamma\backslash\HH^n$ with $\HH^n$ the hyperbolic space of dimension $n$ and $\Gamma$ a discrete group of isometries of $\HH^n$. The geodesic cycles in question are then of the form $Y'\cong\Gamma'\backslash\HH^m$ with $\Gamma'\subseteq\Gamma$ the subgroup of isometries in $\Gamma$ which leave $\HH^m\subseteq\HH^n$ invariant. Let $G=\Isom(\HH^n)\subseteq O(1,n)$ be the full isometry group of $\HH^n$ and $K=G\cap O(n+1)\cong O(n)$ a maximal compact subgroup. Then $\HH^n=G/K$ and hence $Y=\Gamma\backslash G/K$. Accordingly $Y'=\Gamma'\backslash G'/K'$ with
\begin{equation*}
 G'=(O(1,m)\times O(n-m))\cap G, \quad K'=G'\cap K.
\end{equation*}
Write $Y=X/K$ with $X=\Gamma\backslash G$. Then $L^2(Y)\cong L^2(X)^K$ and any automorphic form $\phi$ on $Y$ is the $K$-invariant vector in an irreducible unitary representation $(\pi,\calH)\subseteq L^2(X)$. Here $\calH$ is simply the subrepresentation of $L^2(X)$ generated by the vector $\phi\in L^2(X)^K$. In the same way we associate to each $\phi_j$ an irreducible unitary representation $(\pi_j,\calH_j)\subseteq L^2(X')$, where $X'=\Gamma'\backslash G'$.

Let $V=\calH^\infty\subseteq C^\infty(X)$ and $V_j=\calH_j^\infty\subseteq C^\infty(X')$ denote the spaces of smooth vectors in the representations $\pi$ and $\pi_j$, respectively. For each $j$ we let $\overline{V_j}$ be the smooth vectors in the dual representation of $V_j$, realized as the complex conjugate space of $V_j$ in $L^2(X')$. Then for each $j$ the bilinear form
\begin{equation}
 \ell_j^\aut:V\times\overline{V_j} \to \CC, \quad (v_1,v_2)\mapsto\int_{X'}v_1|_{X'}\cdot v_2,\label{eq:DefAutInvForm}
\end{equation}
is $G'$-invariant. Its value $\ell_j^\aut(\phi,\overline{\phi_j})$ at the spherical vectors is the coefficient $c_j$ we are interested in.

Using the explicit realizations of $\pi$ and $\pi_j$ as representations induced from a parabolic subgroup one can construct model invariant bilinear forms $\ell_j^\mod$ on $V\times\overline{V_j}$ (see Section \ref{sec:ModelForm} for the construction). These forms correspond to $G'$-intertwining operators $\pi|_{G'}\to\pi_j$ which were first studied by Kobayashi--Speh~\cite{KS} (see also \cite{Kob13,MO12}). The space of invariant bilinear forms in this case is generically one-dimensional (\textit{multiplicity one property}). Hence, for $j\gg0$ the form $\ell_j^\aut$ has to be proportional to $\ell_j^\mod$, i.e. there exists a constant $a_j\in\CC$ such that $\ell_j^\aut=a_j\cdot\ell_j^\mod$. Then we have
\begin{equation*}
 c_j = \ell_j^\aut(\phi,\overline{\phi_j}) = a_j\cdot\ell_j^\mod(\phi,\overline{\phi_j}).
\end{equation*}

In Section~\ref{sec:SpecialValue} we calculate the expression $\ell_j^\mod(\phi,\overline{\phi_j})$ explicitly in terms of the eigenvalues of $\phi$ and $\overline{\phi_j}$. Further, for the coefficients $a_j$ we obtain in Section~\ref{sec:PolyBounds} upper bounds by estimating Hermitian forms. We can state the results as follows:

\begin{thmalph}\label{thm:TwoEstimates}
\begin{enumerate}[(1)]
\item There exists a constant $c>0$ such that
\begin{equation*}
 |\ell_j^\mod(\phi,\overline{\phi_j})|^2 \leq c|\lambda_j|^{\frac{n-m-2}{2}}e^{-\pi\sqrt{\lambda_j}}, \qquad j\to\infty.
\end{equation*}
If $V$ is a representation from the unitary principal series then this estimate is sharp.
\item There exists a constant $C>0$ such that for $T\geq1$:
\begin{equation*}
 \sum_{|\lambda_j|\leq T} |a_j|^2 \leq CT^{\frac{n-1}{2}}.
\end{equation*}
\end{enumerate}
\end{thmalph}

Together both estimates imply our main result Theorem~\ref{thm:MainThm}.

\subsection{Concluding remarks}

The strategy of proof is due to Bernstein--Reznikov \cite{BR04} who applied this technique to the case $G=G'\times G'$ and $G'=\PSL(2,\RR)$, embedded in $G$ as the diagonal (see also \cite{BR10,Rez08}). The polynomial estimation of the coefficients $a_j$ in Section~\ref{sec:PolyBounds} pretty much follows their proof. The key new ingredient for our case are the model invariant bilinear forms $\ell_j^\mod$ which correspond to intertwining operators first studied by Kobayashi--Speh~\cite{KS} (see also \cite{Kob13,MO12}). Together with Y. Oshima we generalized these intertwining operators to various symmetric pairs $(G,G')$, see \cite{MOO}. In particular, our construction includes the product situation $G=G'\times G'$ with $G'$ embedded as the diagonal which was studied before (see \cite{CKOP11,Dei06}). For our application the multiplicity one property is crucial and proved in \cite{MOO} (see also \cite{AGRS10,KS,SZ12} for the case $m=n-1$). The heart of the proof of Theorem~\ref{thm:MainThm} is then the calculation of the special values of the model invariant form in Section~\ref{sec:SpecialValue} (see also \cite{KS}). This calculation is radically different from the one used in \cite{BR04}. We use the Fourier transformed realization of principal series representations where the model invariant form corresponds to an intertwining operator between representations of $G'$ and $G$ which is given by integration against a hypergeometric function. This allows us to derive the special values of the form from certain integral formulas for special functions.
\section{Hyperbolic manifolds and geodesic cycles}\label{sec:HyperbolicManifoldsGeodesicCycles}

We recall the geometric setting of hyperbolic manifolds and geodesic cycles.

\subsection{Hyperbolic manifolds}

Let $Y$ be a connected hyperbolic manifold, i.e. $Y$ is a complete connected Riemannian manifold of constant sectional curvature $-1$. Then the universal cover $\widetilde{Y}$ of $Y$ is isomorphic to the hyperbolic space $\HH^n$ where $n=\dim Y$. Let $G=\Isom(\HH^n)$ be the isometry group of $\HH^n$. Then $\Gamma:=\pi_1(Y)\subseteq G$ is a torsion-free discrete subgroup and we can identify $Y\cong\Gamma\backslash\HH^n$.

We realize $\HH^n$ as the one-sheeted hyperboloid
\begin{equation*}
 \HH^n = \{x\in\RR^{n+1}:Q(x)=-1,\,x_1>0\} \subseteq \RR^{n+1},
\end{equation*}
where $Q(x)=-x_1^2+x_2^2+\cdots+x_{n+1}^2$. We endow $\HH^n$ with the metric induced from the Lorentzian metric on $\RR^{n+1}$ corresponding to the quadratic form $Q$. Then $\HH^n$ has constant negative curvature $-1$. In this realization $G=\Isom(\HH^n)$ is a normal subgroup of $O(1,n)$ of index $2$. Here $O(1,n)$ is the subgroup of $\GL(n+1,\RR)$ preserving the quadratic form $Q$. Then $G$ is the group of all $g\in O(1,n)$ such that $(ge_1)_1>0$ with $e_1=(1,0,\ldots,0)\in\HH^n$. The group $G$ has two connected components and acts transitively on $\HH^n$. The stabilizer subgroup of $e_1\in\HH^n$ is the maximal compact subgroup $K=\diag(1,O(n))\cong O(n)$. Hence we can identify $\HH^n\cong G/K$ as Riemannian symmetric spaces. The metric on $\HH^n$ can be viewed as the metric induced from an $\ad$-invariant bilinear form on the Lie algebra $\frakg=\so(1,n)$ of $G$. Namely, $\HH^n=G/K$ carries the Riemannian structure induced from the form
\begin{equation*}
 \kappa(X,Y) = \frac{1}{2}\tr(XY), \qquad X,Y\in\frakg.
\end{equation*}
This form is $\ad$-invariant and non-degenerate on $\frakg$ and hence a scalar multiple of the Killing form of $\frakg$.

For the hyperbolic manifold $Y$ we obtain the identification
\begin{equation*}
 Y \cong \Gamma\backslash G/K.
\end{equation*}
Note that $Y$ is orientable if and only if $\Gamma$ is contained in the identity component $G_0=\SO_0(1,n)$ of $G$. The manifold $Y$ is modelled on $\HH^n=G/K$ and hence inherits the metric of $\HH^n$. The induced Riemannian measure on $Y$ defines the space $L^2(Y)$. Denote the corresponding Laplace--Beltrami operator on $Y$ by $\Delta$. Then $\Delta$ extends to a self-adjoint operator on $L^2(Y)$.

Consider the space $X:=\Gamma\backslash G$ with the natural $G$-action from the right. Since the tangent space of $X$ at $\Gamma e$ is equal to $\frakg$ the space $X$ carries a $G$-invariant pseudo-Riemannian structure induced from the form $\kappa$ on $\frakg$. Again, the corresponding Riemannian measure defines the space $L^2(X)$. Let $\Box$ denote the corresponding Laplace--Beltrami operator on $X$. This operator again extends to a self-adjoint operator on $L^2(X)$. There is also a representation theoretic description of $\Box$. Namely, $\Box$ is up to sign the action of the Casimir element $\calC\in\calU(\frakg)$ with respect to $\kappa$ on $C^\infty(X)$ by the right-regular representation.

Now consider the principal bundle
\begin{equation*}
 X=\Gamma\backslash G \to \Gamma\backslash G/K=Y
\end{equation*}
with structure group $K$. We can identify functions on $Y$ with $K$-invariant functions on $X$. Since $K$ is compact this identification induces an isomorphism $L^2(Y)\cong L^2(X)^K$. If $\phi\in L^2(Y)$ is additionally an eigenfunction of $\Delta$ then the corresponding function $\phi\in L^2(X)^K$ is an eigenfunction of $\Box$ for the same eigenvalue.

\subsection{Geodesic cycles}\label{sec:GeodesicCycles}

We call a totally geodesic submanifold $Y'\subseteq Y$ a \textit{geodesic cycle} (sometimes also referred to as \textit{modular variety}). Let $m=\dim(Y')$. As remarked by Bergeron~\cite[remark after Definition 1]{Ber99} each geodesic cycle is itself a hyperbolic manifold and can be written as $Y'=\Gamma'\backslash\HH^m$. Here one can view $\HH^m$ as a totally geodesic subspace of $\HH^n$ and $\Gamma'=\Gamma\cap G'$ with $G'=\Stab_G(\HH^m)$. By conjugating with an element of $G$ we may assume that $\HH^m\subseteq\HH^n$ is induced by the canonical embedding $\RR^{m+1}\subseteq\RR^{n+1}$ as the first $m+1$ coordinates. Then $G'=\Isom(\HH^m)\times O(n-m)\subseteq O(1,m)\times O(n-m)$.

The subgroup $G'\subseteq G$ is symmetric, i.e. an open subgroup of the fixed point set of an involution of $G$. In fact, $G'=G^\sigma$ with $\sigma$ the involution given by conjugation with the matrix $\diag(\1_{m+1},-\1_{n-m})\in\GL(n+1,\RR)$.

One can construct compact geodesic cycles in the following way. Let $W\subseteq\HH^n$ be a totally geodesic subspace of dimension $m$. Then $W=\HH^n\cap U$ with $U\subseteq\RR^{n+1}$ a linear subspace of dimension $m+1$ on which $Q$ has signature $(m,1)$. Conjugation with the orthogonal reflection at $U$ defines an involution $\sigma$ of $G$ such that $W=G^\sigma/K^\sigma$. The totally geodesic subspace $W$ is called $\Gamma$-compatible if $\sigma\Gamma=\Gamma$. It is proved by Millson--Raghunathan~\cite[Proposition 2.1]{MR81} that if $\Gamma$ is cocompact and $W$ is $\Gamma$-compatible then $\Gamma^\sigma\backslash W=\Gamma^\sigma\backslash G^\sigma/K^\sigma$ is a compact totally geodesic submanifold of $Y$.

It is not clear whether there exist $\Gamma$-compatible totally geodesic subspaces $W\subseteq\HH^n$ for arbitrary discrete subgroups $\Gamma$. However, for $\Gamma$ a group of units over a totally real number field Kudla--Millson~\cite[Section 6]{KM82} construct a large family of $\Gamma$-compatible totally geodesic subspaces of arbitrary dimension $m$ and hence compact geodesic cycles in $Y=\Gamma\backslash G/K$ of arbitrary dimension (see also Millson~\cite[Section 2]{Mil76} for the case $m=n-1$). This provides examples for the setting discussed in this section. Other examples can be constructed similarly using the results by Borel--Harish-Chandra~\cite{BHC62} and Mostow--Tamagawa~\cite{MT62}.
\section{Representation theory - invariant bilinear forms on principal series}

We recall the classification of spherical irreducible unitary representations of $G$ and their relation to eigenfunctions of the Laplacian on $L^2(Y)$. We further describe the recent work of Kobayashi--Speh~\cite{KS} and M\"{o}llers--Oshima--{\O}rsted~\cite{MOO} on invariant bilinear forms on products of principal series representations of $G$ and $G'$.

\subsection{Geometry of the group $G$}

We fix the Cartan involution $\theta$ of $G$ corresponding to the maximal compact subgroup $K=\diag(1,O(n))\subseteq G$. The Lie algebra $\frakg$ of $G$ has the Cartan decomposition $\frakg=\frakk\oplus\frakp$ into the $\pm1$ eigenspaces $\frakk$ and $\frakp$ of $\theta$ where $\frakk$ is the Lie algebra of $K$. Choose the maximal abelian subalgebra $\fraka:=\RR H_0\subseteq\frakp$ spanned by the element
\begin{align*}
 H_0 &:= E_{1,2}+E_{2,1},
\end{align*}
where $E_{ij}$ denotes the $(n+1)\times(n+1)$ matrix with $1$ in the $(i, j)$-entry and $0$ elsewhere. The root system of the pair $(\frakg,\fraka)$ consists only of the roots $\pm\gamma$ where $\gamma\in\fraka_\CC^*$ is defined by $\gamma(H_0):=1$. In what follows we will identify $\fraka_\CC^*\cong\CC$ by means of the isomorphism
\begin{equation}
 \fraka_\CC^*\to\CC,\,\lambda\mapsto \lambda(H_0).\label{eq:IdentificationAC*}
\end{equation}
Then half the sum of all positive roots is given by $\rho=\frac{n-1}{2}$.

Put
\begin{align*}
 \frakn &:= \frakg_{\gamma}, & \overline{\frakn} &:= \frakg_{-\gamma}=\theta\frakn
\intertext{and let}
 N &:= \exp_G(\frakn), & \overline{N} &:= \exp_G(\overline{\frakn})=\theta N
\end{align*}
be the corresponding analytic subgroups of $G$. We introduce the following coordinates on $N$ and $\overline{N}$: For $1\leq j\leq n-1$ let
\begin{align*}
 N_j &:= E_{1,j+2}+E_{2,j+2}+E_{j+2,1}-E_{j+2,2},\\
 \overline{N}_j &:= E_{1,j+2}-E_{2,j+2}+E_{j+2,1}+E_{j+2,2}.
\end{align*}
Then for $x\in\RR^{n-1}$ we put
\begin{equation*}
 n_x := \exp\Big(\sum_{j=1}^{n-1}{x_jN_j}\Big)\in N, \qquad \overline{n}_x := \exp\Big(\sum_{j=1}^{n-1}{x_j\overline{N}_j}\Big)\in\overline{N}.
\end{equation*}
Further put $M:=Z_K(\fraka)$ and $A:=\exp(\fraka)$ and denote by $\frakm$ the Lie algebra of $M$. Then $M=\diag(1,1,O(n-1))\cong O(n-1)$. The group $P:=MAN$ is a parabolic subgroup of $G$ and $\overline{N}P\subseteq G$ is an open dense subset. Let $W:=N_K(\fraka)/Z_K(\fraka)$ be the Weyl group corresponding to $\fraka$. Then $W=\{\1,[w_0]\}$ with $w_0\in K$ acting on $\fraka$ by $\Ad(w_0)|_\fraka=-\id_\fraka$. Hence $w_0^{-1}Nw_0=\overline{N}$ and $w_0^{-1}Pw_0=\overline{P}:=\theta(P)$. On the group level we have the decompositions
\begin{align}
 G &= KAN && \mbox{(Iwasawa decomposition),}\label{eq:IwasawaDecomp}\\
 G &= w_0MAN\cup \overline{N}MAN && \mbox{(Bruhat decomposition)}.\label{eq:BruhatDecomp}
\end{align}
Corresponding to the two decompositions we define functions $H:G\to\fraka$, $\overline{n}:\overline{N}MAN\to\overline{N}$ and $a:\overline{N}MAN\to A$ by
\begin{equation*}
 g \in Ke^{H(g)}N, \qquad g\in \overline{n}(g)Ma(g)N.
\end{equation*}
A straightforward calculation yields
\begin{equation}
 H(\overline{n}_x) = (1+|x|^2)H_0, \qquad x\in\RR^{n-1}.\label{eq:IwasawaProjectionNbar}
\end{equation}

Let $\sigma$ be the involution of $G$ given by conjugation with the matrix $\diag(\1_{m+1},-\1_{n-m})$. Then the symmetric subgroup $G'=G^\sigma$ is given by
\[ G'=(O(1,m)\times O(n-m))\cap G=\Isom(\HH^m)\times O(n-m). \]
It has
\[ K'=K\cap G'=\diag(1,O(m),O(n-m))\cong O(m)\times O(n-m) \]
as maximal compact subgroup.

\subsection{Spherical representations}

For $\nu\in\fraka_\CC^*$ we consider the induced representation (normalized smooth parabolic induction)
\begin{align*}
 I(\nu) &= \Ind_P^G(\1\otimes e^\nu\otimes\1)\\
 &= \{f\in C^\infty(G):f(gman)=a^{-\nu-\rho}f(g)\,\forall\,g\in G,man\in MAN\},
\end{align*}
endowed with the left regular action of $G$:
\begin{equation*}
 (g\cdot f)(x) = f(g^{-1}x), \qquad g,x\in G,\,f\in I(\nu).
\end{equation*}
Then it is well-known that $I(\nu)$ is irreducible and unitarizable if and only if $\nu\in i\RR\cup(-\rho,\rho)$. For these parameters we have $I(\nu)\cong I(-\nu)$. Further, for $\nu=-\rho$ the representation $I(\nu)$ contains the trivial representation as a subrepresentation and for $\nu=\rho$ the trivial representation is the quotient of $I(\nu)$ modulo its unique non-trivial subrepresentation.

All the representations $I(\nu)$ are spherical, the $K$-invariant vectors being the functions which are constant on $K$ and are extended to $G=KP$ according to the transformation law in $I(\nu)$. Conversely, every non-trivial spherical irreducible unitary representation of $G$ is equivalent to $I(\nu)$ for some $\nu\in i\RR\cup(-\rho,\rho)$.

For $\nu\in i\RR$ the invariant Hermitian form $\|\blank\|_\nu$ on $I(\nu)$ is given by
\begin{equation}
 \|f\|_\nu^2 = \int_K |f(k)|^2 \td k,\label{eq:InvariantNormInducedPicture}
\end{equation}
where we denote by $\td k$ the Haar measure on $K$ of mass one. For $\nu\in(0,\rho)$ the invariant Hermitian form $\|\blank\|_\nu$ on $I(\nu)$ is more complicated and we will only describe it in the non-compact picture (see \eqref{eq:InvariantNormFlatPicture2}).

Consider the Casimir element $\calC\in\calU(\frakg)$ with respect to the $\ad$-invariant bilinear form $\kappa$ on $\frakg$. By \cite[Lemma 12.28]{Kna86} the Casimir $\calC$ acts on $I(\nu)$ by the scalar
\begin{equation*}
 |\nu+\rho_\frakt|^2-|\rho_\frakh|^2.
\end{equation*}
Here $\frakh=\frakt+\fraka$ is a Cartan subalgebra of $\frakg$ with $\frakt\subseteq\frakm$ and $\rho_\frakh$ and $\rho_\frakt$ the corresponding half sums of positive roots. It is easy to see that under the identification $\fraka_\CC^*\cong\CC$ as in \eqref{eq:IdentificationAC*} this scalar equals
\begin{equation}
 \nu^2-\rho^2.\label{eq:EigenvalueCasimirPrincipalSeries}
\end{equation}

\subsection{The non-compact picture}

Since $\overline{N}P\subseteq G$ is open dense, a function $f\in I(\nu)$ is already uniquely determined by its restriction to $\overline{N}$. Parameterizing $\overline{N}$ by its Lie algebra $\overline{\frakn}\cong\RR^{n-1}$ we obtain a realization of $I(\nu)$ on smooth functions on $\RR^{n-1}$. More precisely we define for every $f\in I(\nu)$ a function $\calR f\in C^\infty(\RR^{n-1})$ by
\begin{equation*}
 \calR f(x) := f(\overline{n}_x), \qquad x\in\RR^{n-1}.
\end{equation*}
The image of $I(\nu)$ under $\calR$ will be denoted by $J(\nu)$. We have
\begin{equation*}
 C_c^\infty(\RR^{n-1})\subseteq J(\nu)\subseteq C^\infty(\RR^{n-1}).
\end{equation*}
The $G$-action $\pi_\nu$ on $J(\nu)$ is defined by
\begin{equation*}
 \pi_\nu(g)(\calR f) = \calR(g\cdot f).
\end{equation*}
Using \eqref{eq:IwasawaProjectionNbar} we find that the $K$-invariant vector in this realization is given by the function
\begin{equation*}
 \psi_\nu(x) := \calR\1(x) = (1+|x|^2)^{-(\nu+\rho)}, \qquad x\in\RR^{n-1}.
\end{equation*}
Further, the invariant Hermitian form on $J(\nu)$ is for $\nu\in i\RR$ given by
\begin{equation}
 \|f\|_\nu^2 = \frac{\Gamma(2\rho)}{\pi^{\rho}\Gamma(\rho)}\int_{\RR^{n-1}} |f(x)|^2 \td x,\label{eq:InvariantNormFlatPicture}
\end{equation}
and for $\nu\in(0,\rho)$ by
\begin{equation}
 \|f\|_\nu^2 = \frac{\Gamma(2\rho)\Gamma(\nu+\rho)}{\pi^\rho\Gamma(\rho)\Gamma(\nu)}\int_{\RR^{n-1}}\int_{\RR^{n-1}} |x-y|^{2(\nu-\rho)}f(x)\overline{f(y)} \td x\td y.\label{eq:InvariantNormFlatPicture2}
\end{equation}
Note that in this normalization the spherical vector $\psi_\nu$ always has norm one.

The action $\pi_\nu$ can be written in terms of the rational action of $G$ on $\RR^{n-1}$. This rational action is defined by
\begin{equation*}
 \overline{n}_{g\cdot x} := \overline{n}(g\overline{n}_x), \qquad g\in G,x\in\RR^{n-1}.
\end{equation*}
Since $\overline{N}\cdot eP\subseteq G/P$ is open dense with complement only one point, every $g\in G$ acts on every point in $\RR^{n-1}$ except possibly one exception. This defines an action of $G$ on $\RR^{n-1}$ by rational transformations. Now, if a rational transformation $g^{-1}\in G$ is defined at $x\in\RR^{n-1}$ then one can rewrite the action $\pi_\nu$ as
\begin{equation}
 \pi_\nu(g)f(x) = j(g^{-1},x)^{\nu+\rho}f(g^{-1}\cdot x),\label{eq:ActionFlatPicture}
\end{equation}
where $j(g,x)=a(g\overline{n}_x)^{-\gamma}=|\Det Dg(x)|^{\frac{1}{n-1}}$ is the conformal factor. We have
\begin{equation}
 |g\cdot x-g\cdot y|^2 = j(g,x)|x-y|^2j(g,y) \qquad \forall g\in G,x,y\in\RR^{n-1}.\label{eq:NormConformal}
\end{equation}
Change of variables $x\mapsto g\cdot x$ gives the following integral formula:
\begin{equation}
 \int_{\RR^{n-1}} u(g\cdot x)j(g,x)^{2\rho} \td x = \int_{\RR^{n-1}} u(x) \td x.\label{eq:TrafoFormulaFlatPicture}
\end{equation}

The same constructions apply to the group $\Isom(\HH^m)\subseteq O(1,m)$ for which we realize the representations on $\RR^{m-1}$. We view the representations of $\Isom(\HH^m)$ as representations of the symmetric subgroup $G'=\Isom(\HH^m)\times O(n-m)$ by extending them trivially on the compact factor. To distinguish between representations of $G$ and $G'$ we write $(\pi_\nu,J(\nu))$ for the representations of $G$ and $(\pi'_{\nu'},J'(\nu'))$ for those of $G'$. Accordingly, we denote by $\psi_\nu$ and $\psi_{\nu'}'$ the spherical vectors. Further, write $\rho'=\frac{m-1}{2}$ for the half sum of all positive roots for $G'$. Note that the parabolic subgroups in $G$ and $G'$ share the same $A$.

\subsection{Model invariant bilinear forms}\label{sec:ModelForm}

We describe how to construct invariant bilinear forms on $J(\nu)\times J'(\nu')$. For details we refer the reader to \cite{KS} or \cite{MOO} where the corresponding intertwining operators are constructed. Write $x=(x',x'')\in\RR^{m-1}\times\RR^{n-m}=\RR^{n-1}$. For $\Re\nu-\rho\geq\Re\nu'-\rho'\geq0$ the integral
\begin{multline}
 \ell_{\nu,\nu'}^\mod(f_1,f_2) = \\
 \int_{\RR^{n-1}}\int_{\RR^{m-1}} (|x'-y|^2+|x''|^2)^{\nu'-\rho'}|x''|^{(\nu-\rho)-(\nu'-\rho')}f_1(x)f_2(y)\td y\td x\label{eq:ModelFormInFlatPicture}
\end{multline}
converges for all $f_1\in J(\nu)$, $f_2\in J'(\nu')$ and defines a bilinear form on $J(\nu)\times J'(\nu')$ which is invariant under the diagonal action of $G'$ by $\pi_\nu|_{G'}\otimes\pi'_{\nu'}$. This form has a meromorphic continuation in the parameters $\nu,\nu'\in\CC$. The uniqueness result in \cite[Theorem 4.1]{MOO} for intertwining operators immediately implies the following uniqueness result for invariant bilinear forms (see also \cite[Section 3.5.1]{MOO} for the precise relation between invariant bilinear forms and intertwining operators):

\begin{theorem}\label{thm:UniquenessInvariantForms}
For $\nu+\rho\pm\nu'-\rho'\notin(-2\NN_0)$ the space of $G'$-invariant bilinear forms on $J(\nu)\times J'(\nu')$ is at most one-dimensional.
\end{theorem}

In fact, in the case $m=n-1$ Kobayashi--Speh~\cite{KS} show that uniqueness holds for an even larger set of parameters and they find all invariant bilinear forms for arbitrary parameters $\nu,\nu'$ (see also \cite{Kob13}). However, since we also need the case of general $0<m<n$ for which the detailed analysis in \cite{KS} is not available, we use the result in \cite{MOO} instead. We also refer the reader to the recent results on multiplicity one statements by Aizenbud--Gourevitch--Rallis--Schiffmann~\cite{AGRS10} and Sun--Zhu~\cite{SZ12} (see also references therein) which also give the necessary multiplicity one property for the case $m=n-1$.

\subsection{Automorphic representations}\label{sec:AutomorphicRepresentations}

Consider the setting of Section~\ref{sec:HyperbolicManifoldsGeodesicCycles}. Let $\phi\in L^2(Y)\cong L^2(X)^K$ be a non-trivial automorphic form on $Y$ for the eigenvalue $\lambda$. Let $\calH\subseteq L^2(X)$ be the closed subrepresentation of $L^2(X)$ generated by $\phi$ under the right-regular representation of $G$. Then by \cite{GGPS69} the representation $\calH$ is an irreducible spherical unitary representation of $G$. Denote by $V=\calH^\infty\subseteq C^\infty(X)$ its subspace of smooth vectors. Then there exists a $G$-equivariant isometry $J(\nu)\to V$ for some $\nu\in i\RR\cup(-\rho,\rho)$. The Casimir $\calC$ of $\frakg$ acts on $V$ by the negative Laplacian $-\Box$ and hence by the scalar $-\lambda$. On the other hand, by \eqref{eq:EigenvalueCasimirPrincipalSeries} the Casimir acts on $J(\nu)$ by $\nu^2-\rho^2$. Therefore, $\nu$ is up to sign uniquely determined by the equation
\begin{equation*}
 \lambda = \rho^2-\nu^2.
\end{equation*}
To simplify estimates later we always choose $\nu$ such that $\Re\nu\geq0$. We will identify $V\cong J(\nu)$ in what follows. This identifies the invariant Hermitian form on $V$ induced by the $L^2$-inner product with the invariant Hermitian form $\|\blank\|_\nu$ on $J(\nu)$.

Similarly we obtain irreducible spherical unitary representations $\calH_j\subseteq L^2(X')$ of $G'$ for any $j\in\NN$, generated by the basis vectors $\phi_j\in L^2(Y')\cong L^2(X')^{K'}$. For the spaces $V_j=\calH_j^\infty\subseteq C^\infty(X')$ of smooth vectors we consider their complex conjugates $\overline{V_j}\subseteq C^\infty(X')$. They can naturally be identified with the smooth vectors in the representation dual to $\calH_j$. Clearly the Laplacian on $X'$ acts on $\overline{V_j}$ by the same eigenvalue $\lambda_j\in\RR$ as on $V_j$. Therefore we can, as above, identify $\overline{V_j}\cong J'(\nu'_j)$ for some $\nu_j'\in i\RR\cup(-\rho',\rho')$ with
\begin{equation}
 \lambda_j = \rho'^2-\nu_j'^2.\label{eq:RelationLambdaNu}
\end{equation}

Under the identifications $V\cong J(\nu)$  and $\overline{V_j}\cong J'(\nu_j')$ the automorphic forms $\phi\in V$ and $\overline{\phi_j}\in\overline{V_j}$ correspond (up to multiplication with units) to the spherical vectors $\psi_\nu\in J(\nu)$ and $\psi_{\nu_j'}'\in J'(\nu_j')$, respectively. Hence
\begin{equation*}
 |c_j| = |\ell^\aut_j(\psi_\nu,\psi_{\nu_j'}')|,
\end{equation*}
where $c_j$ are the coefficients in the Fourier expansion~\eqref{eq:FourierExpansion} of $\phi$ on $Y'$ and $\ell_j^\aut$ denote the automorphic invariant bilinear forms defined in \eqref{eq:DefAutInvForm}. Write $\ell_j^\mod:=\ell_{\nu,\nu_j'}^\mod$ for the model invariant bilinear form on $J(\nu)\times J'(\nu_j')$. By the uniqueness result in Theorem~\ref{thm:UniquenessInvariantForms} the form $\ell_j^\aut$ is proportional to $\ell_j^\mod$ if the condition
\begin{equation}
 \nu+\rho\pm\nu_j'-\rho'\notin(-2\NN_0)\label{eq:ConditionUniqueness}
\end{equation}
is satisfied. Since $\lambda_j\to\infty$ we also have $|\nu_j'|\to\infty$ by \eqref{eq:RelationLambdaNu}. Note that this implies that for $j\gg0$ we have $\nu_j'\in i\RR$. Hence the condition~\eqref{eq:ConditionUniqueness} is fulfilled for all but finitely many $j$. The estimate in Theorem~\ref{thm:MainThm} is independent of the values of finitely many $b_j$ and therefore we may disregard the finitely many $j$ for which \eqref{eq:ConditionUniqueness} is not true. In what follows we assume that $j\in\NN$ such that $\nu_j'\in i\RR$ and that \eqref{eq:ConditionUniqueness} holds. For these $j$ we obtain proportionality constants $a_j\in\CC$ such that $\ell_j^\aut=a_j\ell_j^\mod$. Hence we find
\begin{equation*}
 |c_j| = |a_j|\cdot|\ell^\mod_{\nu,\nu_j'}(\psi_\nu,\psi_{\nu_j'}')|.
\end{equation*}
In Section~\ref{sec:SpecialValue} we find an explicit formula for the factor $\ell^\mod_{\nu,\nu_j'}(\psi_\nu,\psi_{\nu_j'}')$, proving Theorem~\ref{thm:TwoEstimates}~(1). Estimates for the coefficients $a_j$ are derived in Section~\ref{sec:PolyBounds} which proves Theorem~\ref{thm:TwoEstimates}~(2).
\section{Special value of the model invariant form - exponential bounds}\label{sec:SpecialValue}

This section is devoted to the proof of Theorem~\ref{thm:TwoEstimates}~(1). More precisely, we prove the following explicit formula for $\ell_{\nu,\nu'}^\mod(\psi_\nu,\psi'_{\nu'})$:

\begin{proposition}\label{prop:SpecialValue}
As meromorphic functions in $\nu,\nu'\in\CC$ the following identity holds:
\begin{equation*}
 \ell_{\nu,\nu'}^\mod(\psi_\nu,\psi_{\nu'}') = \frac{\pi^{\rho+\rho'}\Gamma(\rho')\Gamma(\frac{(\nu+\rho)+(\nu'-\rho')}{2})\Gamma(\frac{(\nu+\rho)-(\nu'+\rho')}{2})}{\Gamma(2\rho')\Gamma(\rho-\rho')\Gamma(\nu+\rho)}.
\end{equation*}
In particular, $\ell_{\nu,\nu'}^\mod$ is non-trivial if $\nu\in i\RR\cup(-\rho,\rho)$, $\nu'\in i\RR\cup(-\rho',\rho')$ and $\nu+\rho\pm\nu'-\rho'\notin(-2\NN_0)$.
\end{proposition}

Using Stirling's asymptotics for the Gamma function
\begin{equation*}
 |\Gamma(a+ib)| = \sqrt{2\pi}|b|^{a-\frac{1}{2}}e^{-\frac{\pi}{2}|b|}(1+\calO(|b|^{-1})) \qquad \mbox{as }|b|\to\infty,
\end{equation*}
the identity in Proposition~\ref{prop:SpecialValue} implies the following estimate:

\begin{corollary}\label{cor:EstimateSpecialValue}
For fixed $\nu\in i\RR\cup(-\rho,\rho)$ there exists a constant $c>0$ such that
\begin{equation*}
 |\ell_{\nu,\nu'}^\mod(\psi_\nu,\psi_{\nu'}')| \leq c|\nu'|^{\frac{n-m}{2}-1}e^{-\frac{\pi}{2}|\nu'|}, \qquad \nu'\in i\RR,\,\nu'\to\infty.
\end{equation*}
If $\Re\nu=0$ this estimate is sharp.
\end{corollary}

Now Corollary~\ref{cor:EstimateSpecialValue} implies Theorem~\ref{thm:TwoEstimates}~(1) in view of the relation \eqref{eq:RelationLambdaNu}.\\

The proof of Proposition~\ref{prop:SpecialValue} will be divided into two parts. First we rewrite the integral using the Fourier transform. Then we calculate the resulting integral using several integral formulas for special functions given in Appendix~\ref{sec:IntFormulasSpecialFcts}.

\subsection{The Fourier transform}

Consider the Euclidean Fourier transform $\calF_{\RR^k}:\calS'(\RR^k)\to\calS'(\RR^k)$ given by
\begin{equation*}
 \calF_{\RR^k}u(x) = (2\pi)^{-\frac{k}{2}}\int_{\RR^k} e^{-ix\cdot\xi}u(\xi) \td\xi.
\end{equation*}
It has the following properties (see e.g. \cite{SW71}):
\begin{align*}
 \mbox{(F1)}\quad & \int_{\RR^k}\calF_{\RR^k}u(x)\cdot v(x)\td x=\int_{\RR^k}u(x)\cdot\calF_{\RR^k}v(x)\td x,\\
 \mbox{(F2)}\quad & \calF_{\RR^k}^{-1}u(x)=\calF_{\RR^k}u(-x),\\
 \mbox{(F3)}\quad & \calF_{\RR^k}(u*v)=(2\pi)^{\frac{k}{2}}(\calF_{\RR^k}u)\cdot(\calF_{\RR^k}v),\\
 \mbox{(F4)}\quad & u(x)=\phi(|x|)\quad\Rightarrow\quad\calF_{\RR^k}u(x)=|x|^{-\frac{k-2}{2}}\int_0^\infty J_{\frac{k-2}{2}}(r|x|)\phi(r)r^{\frac{k}{2}}\td r,
\end{align*}
where $J_\alpha(z)$ denotes the classical $J$-Bessel function.

Let
\begin{equation*}
 N:=n-1, \quad M:=m-1.
\end{equation*}
Denote by $\calF_{\RR^N}$ the Fourier transform in the variable $x\in\RR^N$ and by $\calF_{\RR^M}$ and $\calF_{\RR^{N-M}}$ the Fourier transforms in the variables $x'\in\RR^M$ and $x''\in\RR^{N-M}$ where we write $x=(x',x'')\in\RR^M\times\RR^{N-M}$. Then $\calF_{\RR^N}=\calF_{\RR^M}\calF_{\RR^{N-M}}$.

In what follows we use the following abbreviation:
\begin{equation*}
 \alpha = \nu'-\rho', \qquad \beta = \frac{(\nu-\rho)-(\nu'-\rho')}{2}.
\end{equation*}

\begin{lemma}\label{lem:InvariantFormFTPicture}
For $f_1\in J(\nu)$ and $f_2\in J'(\nu')$ we have
\begin{multline*}
 \ell_{\nu,\nu'}^\mod(f_1,f_2) = \frac{2^{2\alpha+2\beta+\frac{M+N}{2}}\pi^{\frac{M}{2}}\Gamma(\alpha+\beta+\frac{N}{2})\Gamma(\beta+\frac{N-M}{2})}{\Gamma(\frac{N-M}{2})\Gamma(-\alpha)} \int_{\RR^N}|x'|^{-2(\alpha+\beta+\frac{N}{2})}\\
 \times{_2F_1}\left(\alpha+\beta+\frac{N}{2},\beta+\frac{N-M}{2};\frac{N-M}{2};-\frac{|x''|^2}{|x'|^2}\right)\calF_{\RR^N}f_1(x)\calF_{\RR^M}f_2(-x')\td x.
\end{multline*}
\end{lemma}

\begin{proof}
Write
\begin{equation*}
 \ell_{\nu,\nu'}^\mod(f_1,f_2) = \int_{\RR^M}\int_{\RR^{N-M}}|x''|^{2\beta}f_1(x)(\phi_{\alpha,x''} *f_2)(x') \td x''\td x',
\end{equation*}
where $\phi_{\alpha,x''}(x')=(|x'|^2+|x''|^2)^\alpha$. Then by (F1) and (F3) we have
\begin{equation*}
 \ell_{\nu,\nu'}^\mod(f_1,f_2) = (2\pi)^{\frac{M}{2}}\int_{\RR^M}\int_{\RR^{N-M}}|x''|^{2\beta}\calF_{\RR^M}^{-1}f_1(x)\calF_{\RR^M}\phi_{\alpha,x''}(x')\calF_{\RR^M}f_2(x') \td x''\td x'.
\end{equation*}
We first calculate $\calF_{\RR^M}\phi_{\alpha,x''}(x')$. Since $\phi_{\alpha,x''}$ is a radial function we find by (F4) that
\begin{equation*}
 \calF_{\RR^M}\phi_{\alpha,x''}(x') = |x'|^{-\frac{M-2}{2}}\int_0^\infty J_{\frac{M-2}{2}}(r|x'|)(r^2+|x''|^2)^\alpha r^{\frac{M}{2}}\td r.
\end{equation*}
Using the integral formula~\eqref{eq:HankelTrafoPartialNorm} we obtain
\begin{equation*}
 = \frac{2^{\alpha+1}}{\Gamma(-\alpha)}\left(\frac{|x''|}{|x'|}\right)^{\alpha+\frac{M}{2}}K_{\alpha+\frac{M}{2}}(|x'|\cdot|x''|).
\end{equation*}
Inserting this we find
\begin{equation*}
 \ell_{\nu,\nu'}^\mod(f_1,f_2) = \frac{2^{\alpha+\frac{M}{2}+1}\pi^{\frac{M}{2}}}{\Gamma(-\alpha)}\int_{\RR^M}\int_{\RR^{N-M}}\varphi_{\alpha,\beta,x'}(x'')\calF_{\RR^M}^{-1}f_1(x)\calF_{\RR^M}f_2(x') \td x''\td x',
\end{equation*}
where $\varphi_{\alpha,\beta,x'}(x'')=|x'|^{-(\alpha+\frac{M}{2})}|x''|^{\alpha+2\beta+\frac{M}{2}}K_{\alpha+\frac{M}{2}}(|x'|\cdot|x''|)$. Using (F1) and (F2) we have
\begin{align*}
 &= \frac{2^{\alpha+\frac{M}{2}+1}\pi^{\frac{M}{2}}}{\Gamma(-\alpha)}\int_{\RR^M}\int_{\RR^{N-M}}\calF_{\RR^{N-M}}\varphi_{\alpha,\beta,x'}(x'')\calF_{\RR^N}^{-1}f_1(x)\calF_{\RR^M}f_2(x') \td x''\td x',\\
 &= \frac{2^{\alpha+\frac{M}{2}+1}\pi^{\frac{M}{2}}}{\Gamma(-\alpha)}\int_{\RR^M}\int_{\RR^{N-M}}\calF_{\RR^{N-M}}\varphi_{\alpha,\beta,x'}(x'')\calF_{\RR^N}f_1(-x)\calF_{\RR^M}f_2(x') \td x''\td x'.
\end{align*}
Now let us calculate $\calF_{\RR^{N-M}}\varphi_{\alpha,\beta,x'}(x'')$. Also $\varphi_{\alpha,\beta,x'}$ is a radial function and by (F4) we obtain
\begin{equation*}
 \calF_{\RR^{N-M}}\varphi_{\alpha,\beta,x'}(x'') = |x'|^{-(\alpha+\frac{M}{2})}|x''|^{-\frac{N-M-2}{2}}\int_0^\infty J_{\frac{N-M-2}{2}}(r|x''|)r^{\alpha+2\beta+\frac{M}{2}}K_{\alpha+\frac{M}{2}}(r|x'|)r^{\frac{N-M}{2}} \td r.
\end{equation*}
With the integral formula~\eqref{eq:HankelTrafoKBessel} this equals
\begin{multline*}
 = \frac{2^{\alpha+2\beta+\frac{N-2}{2}}\Gamma(\alpha+\beta+\frac{N}{2})\Gamma(\beta+\frac{N-M}{2})}{\Gamma(\frac{N-M}{2})} |x'|^{-2(\alpha+\beta+\frac{N}{2})}\\
 \times{_2F_1}\left(\alpha+\beta+\frac{N}{2},\beta+\frac{N-M}{2};\frac{N-M}{2};-\frac{|x''|^2}{|x'|^2}\right).
\end{multline*}
Inserting this into the above formula for $\ell_{\nu,\nu'}^\mod(f_1,f_2)$ gives the claim.
\end{proof}

\begin{remark}
The calculation in Lemma~\ref{lem:InvariantFormFTPicture} can also be found in \cite[Proposition 4.3]{Kob13} and \cite[Section 5]{MO12}. We included it for the sake of completeness.
\end{remark}

\begin{lemma}\label{lem:FourierTrafoSphericalVector}
The Fourier transform of $\psi_\nu(x)$ is given by
\begin{equation*}
 \calF_{\RR^N}\psi_\nu(x) = \frac{1}{2^{\nu+\frac{N-2}{2}}\Gamma(\nu+\frac{N}{2})} |x|^\nu K_\nu(|x|).
\end{equation*}
\end{lemma}

\begin{proof}
Since $\psi_\nu$ is a radial function we can use (F4) to find
\begin{equation*}
 \calF_{\RR^N}\psi_\nu(x) = |x|^{-\frac{N-2}{2}}\int_0^\infty J_{\frac{N-2}{2}}(r|x|)(1+r^2)^{-\nu-\frac{N}{2}}r^{\frac{N}{2}}\td r.
\end{equation*}
Then the integral formula~\eqref{eq:HankelTrafoPartialNorm} and the symmetry $K_\alpha(z)=K_{-\alpha}(z)$ imply the claimed identity.
\end{proof}

\subsection{Proof of Proposition~\ref{prop:SpecialValue}}

By Lemma~\ref{lem:InvariantFormFTPicture} and Lemma~\ref{lem:FourierTrafoSphericalVector} the expression $\ell_{\nu,\nu'}^\mod(\psi_\nu\otimes\psi_{\nu'}')$ is equal to
\begin{multline*}
 \frac{2^{2\alpha+2\beta-\nu-\nu'+2}\pi^{\frac{M}{2}}\Gamma(\alpha+\beta+\frac{N}{2})\Gamma(\beta+\frac{N-M}{2})}{\Gamma(\frac{N-M}{2})\Gamma(-\alpha)\Gamma(\nu+\frac{N}{2})\Gamma(\nu'+\frac{M}{2})} \int_{\RR^N}|x'|^{-2(\alpha+\beta+\frac{N}{2})}\\
 \times{_2F_1}\left(\alpha+\beta+\frac{N}{2},\beta+\frac{N-M}{2};\frac{N-M}{2};-\frac{|x''|^2}{|x'|^2}\right)|x|^\nu K_\nu(|x|)|x'|^{\nu'} K_{\nu'}(|x'|)\td x.
\end{multline*}
Introducing polar coordinates on both $\RR^M$ and $\RR^{N-M}$ we rewrite the integral as
\begin{multline*}
 \vol(S^{M-1})\vol(S^{N-M-1})\int_0^\infty\int_0^\infty r^{\nu'-2(\alpha+\beta+\frac{N}{2})+M-1}s^{N-M-1}(r^2+s^2)^{\frac{\nu}{2}}\\
 \times{_2F_1}\left(\alpha+\beta+\frac{N}{2},\beta+\frac{N-M}{2};\frac{N-M}{2};-\frac{s^2}{r^2}\right) K_\nu(\sqrt{r^2+s^2}) K_{\nu'}(r) \td r\td s.
\end{multline*}
Substituting $t=\frac{s^2}{r^2}$ and using the symmetry $K_\nu(z)=K_{-\nu}(z)$ gives
\begin{multline*}
 =\frac{\vol(S^{M-1})\vol(S^{N-M-1})}{2}\int_0^\infty\int_0^\infty r^{\nu+\nu'-2(\alpha+\beta)-1}t^{\frac{N-M}{2}-1}(1+t)^{\frac{\nu}{2}}\\
 \times{_2F_1}\left(\alpha+\beta+\frac{N}{2},\beta+\frac{N-M}{2};\frac{N-M}{2};-t\right) K_{-\nu}(r\sqrt{1+t}) K_{\nu'}(r) \td r\td t.
\end{multline*}
We use the integral formula~\eqref{eq:IntFormulaProductOf2KBessel} to calculate the integral over $r$ and find
\begin{multline*}
 \frac{2^{\nu+\nu'-2(\alpha+\beta)-3}\Gamma(\nu'-\alpha-\beta)\Gamma(-(\alpha+\beta))\Gamma(\nu+\nu'-\alpha-\beta)\Gamma(\nu-\alpha-\beta)}{\Gamma(\nu+\nu'-2(\alpha+\beta))}\\
 \times\int_0^\infty t^{\frac{N-M}{2}-1}{_2F_1}\left(\alpha+\beta+\frac{N}{2},\beta+\frac{N-M}{2};\frac{N-M}{2};-t\right)\\
 {_2F_1}\left(\nu'-\alpha-\beta,-(\alpha+\beta);\nu+\nu'-2(\alpha+\beta);-t\right) \td t.
\end{multline*}
Next we use the integral representation~\eqref{eq:IntRepForHypergeometric} for the second hypergeometric function and find that the integral is equal to
\begin{multline*}
 \frac{\Gamma(\nu+\nu'-2(\alpha+\beta))}{\Gamma(-(\alpha+\beta))\Gamma(\nu+\nu'-\alpha-\beta)}\int_0^\infty\int_0^1 t^{\frac{N-M}{2}-1}r^{-\alpha-\beta-1}(1-r)^{\nu+\nu'-\alpha-\beta-1}\\
 \times(1+rt)^{\alpha+\beta-\nu'}{_2F_1}\left(\alpha+\beta+\frac{N}{2},\beta+\frac{N-M}{2};\frac{N-M}{2};-t\right) \td r\td t.
\end{multline*}
Using the integral formula~\eqref{eq:IntFormulaHypergeometric0infty} we calculate the integral over $t$ and find
\begin{multline*}
 = \frac{\Gamma(\nu+\nu'-2(\alpha+\beta))}{\Gamma(-(\alpha+\beta))\Gamma(\nu+\nu'-\alpha-\beta)} \frac{\Gamma(\frac{N-M}{2})\Gamma(\nu'+\frac{M}{2})\Gamma(\nu'-\alpha)}{\Gamma(\nu'-\alpha-\beta)\Gamma(\nu'+\beta+\frac{N}{2})}\\
 \times\int_0^1 r^{-\nu'-1}(1-r)^{\nu+\nu'-\alpha-\beta-1}{_2F_1}\left(\nu'+\frac{M}{2},\nu'-\alpha;\nu'+\beta+\frac{N}{2};1-\frac{1}{r}\right) \td r.
\end{multline*}
Substituting $x=\frac{1}{r}-1$ the integral becomes
\begin{equation*}
 \int_0^\infty x^{\nu+\nu'-\alpha-\beta-1}(1+x)^{\alpha+\beta-\nu}{_2F_1}\left(\nu'+\frac{M}{2},\nu'-\alpha;\nu'+\beta+\frac{N}{2};-x\right) \td x.
\end{equation*}
Finally, using once again integral formula~\eqref{eq:IntFormulaHypergeometric0infty} gives (note that $\nu+\nu'-\alpha-\beta=\nu'+\beta+\frac{N}{2}$)
\begin{equation*}
 = \frac{\Gamma(\nu+\nu'-\alpha-\beta)\Gamma(\frac{M}{2})\Gamma(-\alpha)}{\Gamma(\nu-\alpha-\beta)\Gamma(\nu'-\alpha+\frac{M}{2})}.
\end{equation*}
Putting everything together and using $\vol(S^{k-1})=\frac{2\pi^{\frac{k}{2}}}{\Gamma(\frac{k}{2})}$ finally shows the identity claimed in Proposition~\ref{prop:SpecialValue}.
\section{Hermitian forms and positive functionals - polynomial bounds}\label{sec:PolyBounds}

In this section we prove Theorem~\ref{thm:TwoEstimates}~(2). This is done by estimating Hermitian forms. The technique we use is due to Bernstein--Reznikov~\cite{BR04}.

\subsection{Automorphic Hermitian forms}

Let $H^\aut$ denote the Hermitian form on $C^\infty(X)$ given by
\begin{equation*}
 H^\aut(f) = \int_{X'} |f|_{X'}|^2.
\end{equation*}
Further, for each $j\in\NN$ denote by $H_j^\aut$ the Hermitian form on $C^\infty(X)$ given by
\begin{equation*}
 H_j^\aut(f) = \int_{X'} |\pr_j(f|_{X'})|^2,
\end{equation*}
where $\pr_j:L^2(X')\to\overline{\calH_j}$ is the orthogonal projection onto the subspace $\overline{\calH_j}$ of conjugates of $\calH_j$ which restricts to a projection $C^\infty(X')\to\overline{V_j}$. Since the spaces $\overline{\calH_j}$ are pairwise orthogonal subspaces of $L^2(X')$ we have the following basic inequality:
\begin{equation}
 \sum_{j=1}^\infty H_j^\aut \leq H^\aut.\label{eq:HermFormsIneq1}
\end{equation}

\subsection{Hermitian forms associated to invariant bilinear forms}

We can also define Hermitian forms using invariant bilinear forms. For this let $\ell:V\times V_j\to\CC$ be a continuous invariant bilinear form. Then $\ell$ induces an intertwining operator $T_\ell:V\to\overline{V_j}$ where $\overline{V_j}$ is the space of smooth vectors in the dual representation $V_j'$ which can be identified with the complex conjugates of $V_j\subseteq C^\infty(X')$. We define a Hermitian form $H_\ell$ on $V$ by
\begin{equation*}
 H_\ell(f) = \int_{X'} |T_\ell(f)|^2.
\end{equation*}

\begin{lemma}
On $V$ the automorphic Hermitian form $H_j^\aut$ coincides with the form $H_{\ell^\aut_j}$ associated to the automorphic invariant bilinear form $\ell^\aut_j$.
\end{lemma}

\begin{proof}
We just need to check that $T_{\ell_j^\aut}f=\pr_j(f|_{X'})$. For this we calculate for $g\in\overline{V_j}$:
\begin{align*}
 \langle T_{\ell_j^\aut}f,g\rangle_{L^2(X')} &= \ell_j^\aut(f,\overline{g}) = \int_{X'} f|_{X'}\overline{g}\\
 &= \langle f|_{X'},g\rangle_{L^2(X')} = \langle\pr_j(f|_{X'}),g\rangle_{L^2(X')}.\qedhere
\end{align*}
\end{proof}

As for the automorphic Hermitian forms $H_j^\aut$ we denote by $H_j^\mod$ the Hermitian form $H_{\ell_j^\mod}$ associated to the invariant bilinear functional $\ell_j^\mod$. Then with \eqref{eq:HermFormsIneq1} we immediately obtain:

\begin{corollary}
For every $j\in\NN$ we have $H_j^\aut=|a_j|^2 H_j^\mod$ and hence the following inequality for Hermitian forms on $V$ holds:
\begin{equation}
 \sum_{j=1}^\infty |a_j|^2H_j^\mod \leq H^\aut.\label{eq:HermFormsIneq2}
\end{equation}
\end{corollary}

\subsection{Positive functionals}

Let $\calH(V)$ denote the space of Hermitian forms on $V$ and $\calH^+(V)\subseteq\calH(V)$ the cone of non-negative forms. Following Bernstein--Reznikov~\cite{BR04} we call an additive map $\rho:\calH^+(V)\to[0,\infty]$ a \textit{positive functional}. In what follows we will need the following two properties of positive functionals which follow directly from the definition:
\begin{enumerate}[(1)]
\item (Monotonicity) If $H\leq H'$ then $\rho(H)\leq\rho(H')$,
\item (Homogeneity) For $t\geq0$ we have $\rho(tH)=t\rho(H)$.
\end{enumerate}

The group $G$ acts on $V$ and hence on $\calH(V)$. We denote this action by $\Pi$ and extend it to $C_c^\infty(G)$. If $\varphi\in C_c^\infty(G)$, $\varphi\geq0$, then $\Pi(\varphi)$ leaves the $\calH^+(V)$ invariant.

For $u\in V$ and $\varphi\in C_c^\infty(G)$ the map $\rho_{\varphi,u}$ given by
\[ \rho_{\varphi,u}(H) = (\Pi(\varphi)H)(u), \qquad H\in\calH^+(V), \]
is a positive functional.

\subsection{Construction of test functionals}

We now construct certain test functionals which we apply to the inequality~\eqref{eq:HermFormsIneq2} to obtain the desired estimates.

\begin{proposition}\label{prop:TestFunctionals}
There exist $T_0,C,c>0$ such that for every $T\geq T_0$ there exists a positive functional $\rho_T$ on $\calH^+(V)$ with
\begin{align}
 \rho_T(H^\aut) &\leq CT^{n-1},\label{eq:TestFunctionalProp1}\\
 \rho_T(H_j^\mod) &\geq c & \mbox{for }|\nu_j'|\leq T.\label{eq:TestFunctionalProp2}
\end{align}
\end{proposition}

\begin{proof}

We construct a positive functional $\rho_T$ of the form $\rho_T:=\rho_{\varphi,u}$ for certain $\varphi\in C_c^\infty(G)$ and $u\in V$. Here we find $\varphi$ and $u$ in several steps:
\begin{itemize}
\item Choose a point $x_0=(x_0',x_0'')\in\RR^{m-1}\times\RR^{n-m}=\RR^{n-1}$ with
\begin{equation}
 |x_0''| \geq 56.\label{eq:ConstructionProp1}
\end{equation}
\item Since by \eqref{eq:NormConformal} we have $|g\cdot x-g\cdot y|^2=j(g,x)|x-y|^2j(g,y)$ with $j(g,x)$ being smooth in $g$ and $x$, $j(e,x)=1$, there exists a symmetric subset $D_0\subseteq G$ (i.e. $D_0^{-1}=D_0$), which is a relatively compact open neighborhood of the identity, such that every conformal transformation $g\in D_0$ is defined at every $x\in B_{|x_0''|}(x_0)=\{x\in\RR^{n-1}:|x-x_0|<|x_0''|\}$ and we have
\begin{align}
 |g\cdot x_0-x_0| &< |x_0''|/4,\label{eq:ConstructionProp5}\\
 \tfrac{1}{2}|x-y| &\leq |g\cdot x-g\cdot y| \leq 2|x-y| && \mbox{for }x,y\in B_{|x_0''|}(x_0),\label{eq:ConstructionProp4}\\
 \frac{|\nabla_xj(g,x)|}{|j(g,x)|} &\leq \frac{1}{4|\nu-\rho|\cdot|x_0''|} && \mbox{for }x\in B_{|x_0''|}(x_0).\label{eq:ConstructionProp6}
\end{align}
\item Put $D:=D_0K'$ and choose a $K'$-right invariant function $\varphi\in C_c^\infty(G)$ with $\varphi\equiv1$ on $D$. Obviously $\varphi$ is independent of $T$.
\item Next let $T_0>0$ such that
\begin{align}
 & T_0 \geq 4/|x_0''|, \label{eq:ConstructionProp2}\\
 & 6T+4\rho'+2|\nu-\rho+\rho'|+\tfrac{1}{4} \leq 7T & \mbox{for }T\geq T_0.\label{eq:ConstructionProp3}
\end{align}
\item Finally we choose a non-negative function $u_1\in C_c^\infty(\RR^{n-1})\subseteq J(\nu)$ with $\supp u_1\subseteq B_1(x_0)$ and $\int_{\RR^{n-1}}u_1=1$ and put $A:=\|u_1\|_\nu^2>0$. Then the family $(u_T)_{T\geq1}\subseteq C_c^\infty(\RR^{n-1})\subseteq J(\nu)$ given by $u_T(x)=T^nu(Tx)$ has the following properties:
\begin{align}
 & \supp u_T\subseteq B_{T^{-1}}(x_0),\label{eq:ConstructionProp7}\\
 & \int_{\RR^{n-1}}u_T = 1,\label{eq:ConstructionProp8}\\
 & \|u_T\|_\nu^2 = AT^{2(\rho-\Re\nu)} \leq AT^{n-1}.\label{eq:ConstructionProp9}
\end{align}
Here we have used \eqref{eq:InvariantNormFlatPicture} and \eqref{eq:InvariantNormFlatPicture2} for the last property. Note that by \eqref{eq:ConstructionProp2} for $T\geq T_0$ we have $\supp u_T\subseteq B_{|x_0''|/4}(x_0)$.
\end{itemize}

We first prove property~\eqref{eq:TestFunctionalProp1} for the functional $\rho_T=\rho_{\varphi,u}$ with $u=u_T$. We have
\begin{align*}
 \rho_T(H^\aut) &= \int_G \varphi(g)(\Pi(g)H^\aut)(u)\td g = \int_G \varphi(g)H^\aut(\pi(g^{-1})u)\td g\\
 &= \int_G \int_{X'} \varphi(g)|u(g\cdot x)|^2 \td\mu_{X'}(x)\td g,
\end{align*}
where $\td\mu_{X'}$ denotes the Riemannian measure on $X'$. Since $G\cdot X'=X$ the map 
\begin{equation*}
 C_c(X)\to\CC,\,f\mapsto\int_G \int_{X'} \varphi(g)f(g\cdot x) \td\mu_{X'}(x)\td g
\end{equation*}
defines a smooth finite measure on $X$. Hence this measure must be bounded by a constant $B>0$ times the $G$-invariant measure $\td\mu_X$, the constant only depending on $\varphi$ which in turn does not depend on $T$ or $T_0$ but only on the choice of $x_0$ and $D_0$. It follows that
\begin{equation*}
 \rho_T(H^\aut) \leq B\int_X |u|^2 \td\mu_X = B\|u\|_\nu^2 \leq CT^{n-1}
\end{equation*}
by \eqref{eq:ConstructionProp9} with $C=AB$ which proves \eqref{eq:TestFunctionalProp1}.

Now let us calculate $\rho_T(H_j^\mod)$. We use the fact that for the spherical principal series $J'(\nu_j')$ the invariant norm is given by (see \eqref{eq:InvariantNormInducedPicture})
\begin{equation*}
 \|f\|_{\nu_j'}^2 =\int_{K'} |(\pi'_{\nu_j'}(k')f)(0)|^2 \td k'.
\end{equation*}
Using the intertwining property of $T_{\ell_j^\mod}$ and the $K'$-right invariance of $\varphi$ we obtain
\begin{align*}
 \rho_T(H_j^\mod) &= \int_G \varphi(g)(\Pi(g)H_j^\mod)(u) \td g = \int_G \varphi(g)H_j^\mod(\pi(g^{-1})u) \td g\\
 &= \int_G \varphi(g) \|T_{\ell_j^\mod}\pi(g^{-1})u\|^2_{\nu_j'} \td g\\
 &= \int_G \int_{K'} \varphi(g) |(\pi'_{\nu_j'}(k')T_{\ell_j^\mod}\pi(g^{-1})u)(0)|^2 \td k' \td g\\
 &= \int_G \int_{K'} \varphi(g) |(T_{\ell_j^\mod}\pi(k'g^{-1})u)(0)|^2 \td k' \td g\\
 &= \int_G \varphi(g) |(T_{\ell_j^\mod}\pi(g^{-1})u)(0)|^2 \td g.
\end{align*}
Since $\varphi\equiv1$ on $D_0$ and $\varphi\geq0$ it suffices to show that $|(T_{\ell_j^\mod}\pi(g^{-1})u)(0)|^2$ is on $D_0$ bounded below by a universal constant for $|\nu_j'|\leq T$.

In the flat picture the intertwiner $T_{\ell_j^\mod}$ takes by \eqref{eq:ModelFormInFlatPicture} the form
\begin{equation*}
 T_{\ell_j^\mod}f(0) = C_{m,n}\cdot\int_{\RR^{n-1}} |x|^{\nu_j'-\rho'}|x''|^{\frac{\nu-\nu_j'-(\rho-\rho')}{2}}f(x)\td x,
\end{equation*}
the constant $C_{m,n}>0$ only depending on the dimensions $m$ and $n$. The action $\pi=\pi_\nu$ is in the flat picture given by \eqref{eq:ActionFlatPicture} and hence
\begin{equation*}
 |(T_{\ell_j^\mod}\pi(g^{-1})u)(0)| = C_{m,n}\cdot\left|\int_{\RR^{n-1}} |x|^{\nu_j'-\rho'}|x''|^{\frac{\nu-\nu_j'-(\rho-\rho')}{2}}j(g,x)^{\nu+\rho}u(g\cdot x) \td x\right|.
\end{equation*}
First note that $g\cdot x\in\supp u\subseteq B_{T^{-1}}(x_0)$ implies $x\in g^{-1}\cdot B_{T^{-1}}(x_0)\subseteq B_{2T^{-1}}(g^{-1}x_0)\subseteq B_{3|x_0''|/4}(x_0)$ by \eqref{eq:ConstructionProp5}, \eqref{eq:ConstructionProp4} and \eqref{eq:ConstructionProp2}, so the integral equals
\begin{equation*}
 \int_{B_{2T^{-1}}(g^{-1}x_0)} |x|^{\nu_j'-\rho'}|x''|^{\frac{\nu-\nu_j'-(\rho-\rho')}{2}}j(g,x)^{\nu+\rho}u(g\cdot x) \td x.
\end{equation*}
To find a universal lower bound for this expression we use the following fact (cf. \cite[Section 5.3]{BR04}):

\begin{fact}\label{fct:IntLowerBd}
Let $(X,\mu)$ be a measure space and $f,\phi:X\to\CC$ two measurable functions. Assume $\phi\geq0$ with $\int_X\phi\td\mu=1$ and $\sup_{x,y\in X}|f(x)-f(y)|\leq\frac{1}{2}\sup_{x\in X}|f(x)|$. Then
\begin{equation*}
 \left|\int_Xf\phi\td\mu\right| \geq \frac{1}{2}\sup_{x\in X}|f(x)|.
\end{equation*}
\end{fact}

We use this fact with $f_g(x):=|x|^{\nu_j'-\rho'}|x''|^{\frac{\nu-\nu_j'-(\rho-\rho')}{2}}j(g,x)^{\nu-\rho}$ and $\phi_g(x)=j(g,x)^{2\rho}u(g\cdot x)$. Clearly $\phi_g\geq0$ and by \eqref{eq:TrafoFormulaFlatPicture} and \eqref{eq:ConstructionProp8} we have $\int_{\RR^{n-1}}\phi_g(x)\td x=1$. To prove the assumption on $f_g$ in Fact~\ref{fct:IntLowerBd} we use the mean value theorem. For the gradient of $f_g(x)$ we find
\begin{equation*}
 \nabla f_g(x) = \left[(\nu_j'-\rho')\frac{x}{|x|^2}+\frac{\nu-\nu_j'-(\rho-\rho')}{2}\frac{x''}{|x''|^2}+(\nu-\rho)\frac{\nabla_xj(g,x)}{j(g,x)}\right]f_g(x).
\end{equation*}
Using \eqref{eq:ConstructionProp6} we can estimate the gradient by
\begin{multline*}
 \sup_{x\in B_{2T^{-1}}(g^{-1}x_0)}|\nabla f_g(x)| \leq \sup_{x\in B_{2T^{-1}}(g^{-1}x_0)}\left[\frac{1}{|x''|}\left(\tfrac{3}{2}|\nu_j'|+\rho'+|\tfrac{\nu-(\rho-\rho')}{2}|\right)+\frac{1}{4|x_0''|}\right]\\
 \times\sup_{x\in B_{2T^{-1}}(g^{-1}x_0)}|f_g(x)|
\end{multline*}
Using \eqref{eq:ConstructionProp5} and \eqref{eq:ConstructionProp2} we find that $x\in B_{2T^{-1}}(g^{-1}x_0)$ implies $|x''|\geq|x_0''|/4$ and hence
\begin{multline*}
 \sup_{x\in B_{2T^{-1}}(g^{-1}x_0)}\left[\frac{1}{|x''|}\left(\tfrac{3}{2}|\nu_j'|+\rho'+|\tfrac{\nu-(\rho-\rho')}{2}|\right)+\frac{1}{4|x_0''|}\right]\\
 \leq \frac{1}{|x_0''|}\left(6|\nu_j'|+4\rho'+2|\nu-(\rho-\rho')|+\tfrac{1}{4}\right).
\end{multline*}
For $|\nu_j'|\leq T$ the last term can by \eqref{eq:ConstructionProp1} and \eqref{eq:ConstructionProp3} be estimated by $T/8$. Hence we obtain for $g\in D_0$
\begin{align*}
 \sup_{x,y\in B_{2T^{-1}}(g^{-1}x_0)}|f_g(x)-f_g(y)| &\leq 4T^{-1}\cdot\sup_{x\in B_{2T^{-1}}(g^{-1}x_0)} |\nabla f_g(x)|\\
 &\leq \frac{1}{2}\sup_{x\in B_{2T^{-1}}(g^{-1}x_0)}|f_g(x)|
\end{align*}
Therefore $f_g$ satisfies the assumptions in Fact~\ref{fct:IntLowerBd} and we obtain
\begin{equation*}
 \rho_T(H_j^\mod) \geq \frac{1}{2}C_{m,n}^2\vol(D_0)\left(\sup_{x\in B_{2T^{-1}}(g^{-1}x_0)}|f_g(x)|\right)^2.
\end{equation*}
By \eqref{eq:ConstructionProp5} and \eqref{eq:ConstructionProp2} we have $B_{2T^{-1}}(g^{-1}x_0)\subseteq B_{3|x_0''|/4}(x_0)$. But for $x\in B_{3|x_0''|/4}(x_0)$ we have the following lower bound:
\begin{equation*}
 |f_g(x)| \geq \sup_{x\in B_{3|x_0''|/4}(x_0),\,g\in D_0}|x|^{-2\rho'}|x''|^{\frac{\Re\nu-\rho+\rho'}{2}}j(g,x)^{\Re\nu-\rho} > 0
\end{equation*}
since $B_{3|x_0''|/4}(x_0)$ and $D_0$ are relatively compact. This bound only depends on the choice of $x_0$ and $D_0$ and not on $T$ or $T_0$. Therefore the proof is complete.
\end{proof}

\subsection{Proof of Theorem~\ref{thm:TwoEstimates}~(2)}

We apply the test functionals $\rho_T$ from Proposition~\ref{prop:TestFunctionals} to the inequality~\eqref{eq:HermFormsIneq2}. Hence there exists $T_0,C,c>0$ such that for every $T\geq T_0$ we have
\begin{equation*}
 c\sum_{|\nu_j'|\leq T}{|a_j|^2} \leq CT^{n-1}.
\end{equation*}
The parameter $\nu_j'$ is related to the eigenvalue $\lambda_j$ by \eqref{eq:RelationLambdaNu} and hence this estimate implies Theorem~\ref{thm:TwoEstimates}~(2).

\appendix

\section{Integral formulas for special functions}\label{sec:IntFormulasSpecialFcts}

\subsection{Bessel functions}

Let $J_\alpha(z)$ and $K_\alpha(z)$ denote the classical $J$- and $K$-Bessel functions. The following integral formulas hold:

For $-1<\Re\nu<\Re(2\mu+\frac{3}{2})$, $a,b>0$ we have (see \cite[formula~6.565~(4)]{GR07})
\begin{equation}
 \int_0^\infty \frac{J_\nu(bx)x^{\nu+1}}{(x^2+a^2)^{\mu+1}}\td x = \frac{a^{\nu-\mu}b^\mu}{2^\mu\Gamma(\mu+1)}K_{\nu-\mu}(ab).\label{eq:HankelTrafoPartialNorm}
\end{equation}

For $\Re(a\pm ib)>0$, $\Re(\nu-\lambda+1)>|\Re\mu|$ we have (see \cite[formula~6.576~(3)]{GR07})
\begin{multline}
 \int_0^\infty x^{-\lambda}K_\mu(ax)J_\nu(bx) \td x = \frac{b^\nu\Gamma(\frac{\nu-\lambda+\mu+1}{2})\Gamma(\frac{\nu-\lambda-\mu+1}{2})}{2^{\lambda+1}a^{\nu-\lambda+1}\Gamma(\nu+1)}\\
 \times{_2F_1}\left(\frac{\nu-\lambda+\mu+1}{2},\frac{\nu-\lambda-\mu+1}{2};\nu+1;-\frac{b^2}{a^2}\right).\label{eq:HankelTrafoKBessel}
\end{multline}

For $\Re\sigma>|\Re\mu|+|\Re\nu|$, $\Re(a+b)>0$ we have (see \cite[formula 10.3~(49)]{EMOT54})
\begin{multline}
 \int_0^\infty K_\mu(ax)K_\nu(bx)x^{\sigma-1}\td x = \frac{2^{\sigma-3}b^\nu\Gamma(\frac{\sigma+\mu+\nu}{2})\Gamma(\frac{\sigma-\mu+\nu}{2})\Gamma(\frac{\sigma+\mu-\nu}{2})\Gamma(\frac{\sigma-\mu-\nu}{2})}{a^{\nu+\sigma}\Gamma(\sigma)}\\
 {_2F_1}\left(\frac{\sigma+\mu+\nu}{2},\frac{\sigma-\mu+\nu}{2};\sigma;1-\frac{b^2}{a^2}\right).\label{eq:IntFormulaProductOf2KBessel}
\end{multline}

\subsection{Hypergeometric function}

Further, let ${_2F_1}(a,b;c;z)$ denote the classical hypergeometric function. Then the following integral formulas hold:

For $\Re c>\Re b>0$, $x\notin(1,\infty)$ we have (see \cite[Theorem 2.2.1]{AAR99})
\begin{equation}
 {_2F_1}(a,b;c;x) = \frac{\Gamma(c)}{\Gamma(b)\Gamma(c-b)}\int_0^1t^{b-1}(1-t)^{c-b-1}(1-xt)^{-a}\td t.\label{eq:IntRepForHypergeometric}
\end{equation}

For $\Re\gamma>0$, $\Re(\alpha-\gamma+\sigma)>0$, $\Re(\beta-\gamma+\sigma)>0$, $\arg(z)<\pi$ we have (see \cite[equation 7.512~(10)]{GR07})
\begin{multline}
 \int_0^\infty x^{\gamma-1}(x+z)^{-\sigma}{_2F_1}(\alpha,\beta;\gamma;-x)\td x = \frac{\Gamma(\gamma)\Gamma(\alpha-\gamma+\sigma)\Gamma(\beta-\gamma+\sigma)}{\Gamma(\sigma)\Gamma(\alpha+\beta-\gamma+\sigma)}\\
 {_2F_1}(\alpha-\gamma+\sigma,\beta-\gamma+\sigma;\alpha+\beta-\gamma+\sigma;1-z).\label{eq:IntFormulaHypergeometric0infty}
\end{multline}

%

\bibliographystyle{amsplain}
\bibliography{bibdb}

\def\cprime{$'$}
\providecommand{\bysame}{\leavevmode\hbox to3em{\hrulefill}\thinspace}
\providecommand{\MR}{\relax\ifhmode\unskip\space\fi MR }
\providecommand{\MRhref}[2]{%
  \href{http://www.ams.org/mathscinet-getitem?mr=#1}{#2}
}
\providecommand{\href}[2]{#2}
\begin{thebibliography}{10}

\bibitem{AGRS10}
A.~Aizenbud, D.~Gourevitch, S.~Rallis, and G.~Schiffmann, \emph{Multiplicity
  one theorems}, Ann. of Math. (2) \textbf{172} (2010), no.~2, 1407--1434.

\bibitem{AAR99}
G.~E. Andrews, R.~Askey, and R.~Roy, \emph{Special functions}, Encyclopedia of
  Mathematics and its Applications, vol.~71, Cambridge University Press,
  Cambridge, 1999.

\bibitem{Ber99}
N.~Bergeron, \emph{Sur l'homologie de cycles g\'eod\'esiques dans des
  vari\'et\'es hyperboliques compactes}, C. R. Acad. Sci. Paris S\'er. I Math.
  \textbf{328} (1999), no.~9, 783--788.

\bibitem{Ber06}
\bysame, \emph{Propri\'et\'es de {L}efschetz automorphes pour les groupes
  unitaires et orthogonaux}, M\'em. Soc. Math. Fr. (N.S.) (2006), no.~106.

\bibitem{BR04}
J.~Bernstein and A.~Reznikov, \emph{Estimates of automorphic functions}, Mosc.
  Math. J. \textbf{4} (2004), no.~1, 19--37, 310.

\bibitem{BR10}
\bysame, \emph{Subconvexity bounds for triple {$L$}-functions and
  representation theory}, Ann. of Math. (2) \textbf{172} (2010), no.~3,
  1679--1718.

\bibitem{BHC62}
A.~Borel and Harish-Chandra, \emph{Arithmetic subgroups of algebraic groups},
  Ann. of Math. (2) \textbf{75} (1962), 485--535.

\bibitem{BGT07}
N.~Burq, P.~G{\'e}rard, and N.~Tzvetkov, \emph{Restrictions of the
  {L}aplace--{B}eltrami eigenfunctions to submanifolds}, Duke Math. J.
  \textbf{138} (2007), no.~3, 445--486.

\bibitem{CKOP11}
J.-L. Clerc, T.~Kobayashi, B.~{\O}rsted, and M.~Pevzner, \emph{Generalized
  {B}ernstein--{R}eznikov integrals}, Math. Ann. \textbf{349} (2011), no.~2,
  395--431.

\bibitem{Dei06}
A.~Deitmar, \emph{Invariant triple products}, Int. J. Math. Math. Sci. (2006),
  Art. ID 48274, 22.

\bibitem{EMOT54}
A.~Erd{\'e}lyi, W.~Magnus, F.~Oberhettinger, and F.~G. Tricomi, \emph{Tables of
  integral transforms. {V}ol. {II}}, McGraw-Hill Book Company, Inc., New York,
  1954.

\bibitem{GGPS69}
I.~M. Gel{\cprime}fand, M.~I. Graev, and I.~I. Pyatetskii-Shapiro,
  \emph{Representation theory and automorphic functions}, Translated from the
  Russian by K. A. Hirsch, W. B. Saunders Co., Philadelphia, Pa., 1969.

\bibitem{GR07}
I.~S. Gradshteyn and I.~M. Ryzhik, \emph{Table of integrals, series, and
  products}, seventh ed., Elsevier/Academic Press, Amsterdam, 2007.

\bibitem{GP92}
B.~H. Gross and D.~Prasad, \emph{On the decomposition of a representation of
  {${\rm SO}_n$} when restricted to {${\rm SO}_{n-1}$}}, Canad. J. Math.
  \textbf{44} (1992), no.~5, 974--1002.

\bibitem{II10}
A.~Ichino and T.~Ikeda, \emph{On the periods of automorphic forms on special
  orthogonal groups and the {G}ross--{P}rasad conjecture}, Geom. Funct. Anal.
  \textbf{19} (2010), no.~5, 1378--1425.

\bibitem{Kna86}
A.~W. Knapp, \emph{Representation theory of semisimple groups}, Princeton
  Mathematical Series, vol.~36, Princeton University Press, Princeton, NJ,
  1986, An overview based on examples.

\bibitem{Kob13}
T.~Kobayashi, \emph{{$F$}-method for symmetry breaking operators},  (2013), to
  appear in Differential Geom.\ Appl., available at
  \href{http://dx.doi.org/10.1016/j.difgeo.2013.10.003}{DOI:10.1016/j.difgeo.2%
013.10.003}.

\bibitem{KO98}
T.~Kobayashi and T.~Oda, \emph{A vanishing theorem for modular symbols on
  locally symmetric spaces}, Comment. Math. Helv. \textbf{73} (1998), no.~1,
  45--70.

\bibitem{KS}
T.~Kobayashi and B.~Speh, \emph{Symmetry breaking for representations of rank
  one orthogonal groups},  (2013), preprint, available at
  \href{http://arxiv.org/abs/1310.3213}{arXiv:1310.3213}.

\bibitem{KM82}
S.~S. Kudla and J.~J. Millson, \emph{Geodesic cyclics and the {W}eil
  representation. {I}. {Q}uotients of hyperbolic space and {S}iegel modular
  forms}, Compositio Math. \textbf{45} (1982), no.~2, 207--271.

\bibitem{Mil76}
J.~J. Millson, \emph{On the first {B}etti number of a constant negatively
  curved manifold}, Ann. of Math. (2) \textbf{104} (1976), no.~2, 235--247.

\bibitem{MR81}
J.~J. Millson and M.~S. Raghunathan, \emph{Geometric construction of cohomology
  for arithmetic groups. {I}}, Proc. Indian Acad. Sci. Math. Sci. \textbf{90}
  (1981), no.~2, 103--123.

\bibitem{MO12}
J.~M\"{o}llers and Y.~Oshima, \emph{Restriction of complementary series
  representations of {$O(1,N)$} to symmetric subgroups},  (2012), preprint,
  available at \href{http://arxiv.org/abs/1209.2312}{arXiv:1209.2312}.

\bibitem{MOO}
J.~M\"{o}llers, Y.~Oshima, and B.~{\O}rsted, \emph{{K}napp--{S}tein type
  intertwining operators for symmetric pairs},  (2013), preprint, available at
  \href{http://arxiv.org/abs/1309.3904}{arXiv:1309.3904}.

\bibitem{MT62}
G.~D. Mostow and T.~Tamagawa, \emph{On the compactness of arithmetically
  defined homogeneous spaces}, Ann. of Math. (2) \textbf{76} (1962), 446--463.

\bibitem{Rez04}
A.~Reznikov, \emph{Estimates of geodesic restrictions of eigenfunctions on
  hyperbolic surfaces and representation theory},  (2004), preprint, available
  at \href{http://arxiv.org/abs/math/0403437}{arXiv:math/0403437}.

\bibitem{Rez08}
\bysame, \emph{Rankin--{S}elberg without unfolding and bounds for spherical
  {F}ourier coefficients of {M}aass forms}, J. Amer. Math. Soc. \textbf{21}
  (2008), no.~2, 439--477.

\bibitem{Rez13}
\bysame, \emph{A uniform bound for geodesic periods of eigenfunctions on
  hyperbolic surfaces},  (2013), to appear in Forum Math.

\bibitem{Sch10}
J.~Schwermer, \emph{Geometric cycles, arithmetic groups and their cohomology},
  Bull. Amer. Math. Soc. (N.S.) \textbf{47} (2010), no.~2, 187--279.

\bibitem{SW71}
E.~M. Stein and G.~Weiss, \emph{Introduction to {F}ourier analysis on
  {E}uclidean spaces}, Princeton University Press, Princeton, N.J., 1971,
  Princeton Mathematical Series, No. 32.

\bibitem{SZ12}
B.~Sun and C.-B. Zhu, \emph{Multiplicity one theorems: the {A}rchimedean case},
  Ann. of Math. (2) \textbf{175} (2012), no.~1, 23--44.

\bibitem{Wat02}
T.~C. Watson, \emph{Rankin triple products and quantum chaos}, ProQuest LLC,
  Ann Arbor, MI, 2002, Thesis (Ph.D.)--Princeton University.

\end{thebibliography}

\end{document}